\documentclass[11pt]{article}

\usepackage[a4paper,margin=1in]{geometry}
\usepackage{amsmath,amssymb,amsthm,mathtools}
\usepackage{bm}
\usepackage{enumitem}
\usepackage{microtype}
\usepackage{hyperref}
\usepackage{mathrsfs}
\usepackage[nameinlink,capitalize]{cleveref}
\usepackage{booktabs}
\usepackage{latexsym}
\usepackage{amssymb,amsbsy,amsmath,amsfonts,amssymb,amscd}
\usepackage{graphicx,color}
\usepackage{float,url}
\usepackage{cancel}
\usepackage{mathtools}
\hypersetup{colorlinks=true,linkcolor=blue,citecolor=blue,urlcolor=blue,
  pdftitle={Deriving Subdiffusion Equations from General Renewal--Jump Dynamics},
  pdfauthor={Benoit Perthame and Min Tang}}
\allowdisplaybreaks

\newcommand{\commentout}[1]{}

\newcommand {\Chi} {{\bf \raise 2pt \hbox{$\chi$}} }

\newcommand{\beq}{\begin{equation}}
\newcommand{\eeq}{\end{equation}}

\newcommand{\dis}{\displaystyle}

\newcommand {\caL} { {\mathcal L} }

\newtheorem{theorem}{Theorem}[section]
\newtheorem{proposition}[theorem]{Proposition}
\newtheorem{lemma}[theorem]{Lemma}
\newtheorem{corollary}[theorem]{Corollary}
\theoremstyle{remark}
\newtheorem{remark}[theorem]{Remark}

\newcommand{\R}{\mathbb{R}}
\newcommand{\eps}{\varepsilon}
\newcommand{\Lap}{\mathcal{L}}
\newcommand{\J}{\mathcal{J}}

\newcommand{\dd}{\,\mathrm{d}}

\newcommand{\CapD}{\prescript{\mathrm C}{}{D}_t^\alpha}
\newcommand{\RLD}{\prescript{\mathrm{RL}}{}{D}_t^\alpha}

\title{Deriving subdiffusion equations from general Renewal--Jump dynamics}
\author{Beno\^it Perthame\thanks{Sorbonne Universit\'e, CNRS, Universit\'e de Paris, Inria, Laboratoire Jacques-Louis Lions UMR 7598, F-75005 Paris, France. Email: \texttt{Benoit.Perthame@sorbonne-universite.fr}.}
\and Min Tang\thanks{School of Mathematical Sciences and Institute of Natural Sciences, MOE-LSC, CMA-Shanghai, Shanghai Jiao Tong University, Shanghai, China. Email: \texttt{tangmin@sjtu.edu.cn}.}}
\date{August 19, 2026}

\begin{document}
\maketitle

\begin{abstract}
Subdiffusion occurs when long trapping or residence times of particles slow down spatial transport and produce a mean squared displacement proportional to $t^\alpha$, with~$0<\alpha<1$. We develop a general renewal--jump framework that combines the internal trapping dynamics with the reinjection of particles with spatial jumps. The internal trapping dynamics enters the macroscopic limit through the small-frequency behavior of a resolvent which limiting solution is not integrable. A growth mass condition determines the parameter~$\alpha$.
Then the spatial density converges on the time scale~$\eps^{-2/\alpha}$ to a time-fractional diffusion equation. This criterion provides a common derivation for models with very different internal mechanisms.

We first establish the abstract limit by Laplace transform in a general setting. We then show that the classical age-structured renewal model is a direct instance of the framework and apply the same criterion to an internal pathway model governed by a degenerate elliptic operator.
Both yield the subdiffusion equation and its correct initial condition. The approach clarifies which microscopic property produces subdiffusion, how the anomalous exponent determines the macroscopic scaling, and why the initial distribution enters the limit through the initial spatial mass under the stated preparation assumptions.
\end{abstract}



\section{Introduction}

Subdiffusion is characterized by a mean squared displacement
\[
 \mathbb E|X(t)-X(0)|^2\sim t^\alpha,
 \qquad 0<\alpha<1,
\]
and has been observed in fluids, plasma physics, disordered media, economics and biological systems. Examples include lipid granules in living yeast cells \cite{sub4,sub7}, insulin granules undergoing intermittent microtubule binding and unbinding \cite{sub8}, and hydration water moving on heterogeneous protein surfaces \cite{Tan2018,Li2022}.  In these examples, a typical behavior is long trapping or residence times of particles between spatial jumps. We refer to \cite{Meroz2015,METZLER20001} for broad discussions of subdiffusive mechanisms.

A standard macroscopic description uses a fractional derivative in time. The Caputo operators \cite{Caputo1999} has different writings,
\begin{align}
 \CapD f(t)
 &=\frac{1}{\Gamma(1-\alpha)}
   \int_0^t\frac{\partial_\tau f(\tau)}{(t-\tau)^\alpha}\,\dd\tau
 =\RLD\bigl(f-f(0)\bigr).\label{eq:Caputo}
\end{align}
The typical subdiffusion equation is
\begin{equation}
 \CapD\rho-\kappa\Delta\rho=0,
 \qquad \rho(0,x)=\rho^{\rm in}(x).
 \label{eq:standard-fractional}
\end{equation}
The nonlocal time kernel records the history of trapping events.

The continuous-time random walk (CTRW) gives a classical microscopic interpretation. Particles alternate between spatial jumps and random waiting times. If the waiting-time density behaves as $t^{-1-\alpha}$, then the mean waiting time diverges and the ensemble-averaged motion is subdiffusive. Earlier works connected CTRWs to fractional Fokker--Planck equations through Green functions or generalized master equations \cite{PhysRevE.61.132,SIAP2015}. In such formulations, however, the heavy-tailed waiting-time law is usually prescribed. Such phenomena is usually explained by the fact that particles are heterogeneous, i.e. different particle behave differently, one can model this fact by adding an additional variable to the jump model that takes into account the heterogenity of the cells. For a model with internal state such as age or intracellular noise, this leaves open a mechanistic question: which internal dynamics generates the power-law residence time and, more generally, which microscopic operator is actually needed for the fractional limit?

The purpose of this paper is to address this question. We consider particles with an internal state $a$ describing, for example, the time spent in a trap or a chemical pathway state. The internal state evolves between renewals; a renewal occurs at rate $d(a)$; and each renewal is accompanied by a small spatial jump. Normal and superdiffusive limits for related nonclassical transport models were studied in \cite{GoudonFrank_2010,FrankSun_2018}. A common assumption in the normal-diffusion analysis is that the internal equilibrium is integrable. In normal diffusion, an integrable equilibrium allows direct averaging over the internal variable. In the subdiffusive regime, the equilibrium has infinite mass and this averaging fails and the mean residence time is infinite. We provide a general framework of deriving subdiffusion equation when the equilibrium has infinite mass, no matter the details of how the internal state evolves.

In particular, two quantities are of central importance in this general framework. The first is the normalized total mass, given that the equilibrium distribution has infinite mass. If a normalized small $\lambda$ is introduced into the equilibrium equation, the resolvent profile $Q_\lambda$ of the internal dynamics satisfies
\[
m(\lambda)=\int Q_\lambda(a)\,\mathrm{d}a\sim\gamma\lambda^{\alpha-1},
\]
which provides the general signature of subdiffusion. As $\lambda\to 0$, $Q_\lambda$ approaches the original equilibrium and $m(\lambda)\to\infty$.
The second key quantity is the renewal flux
\[
r(\lambda)=\int d(a)Q_\lambda(a)\,\mathrm{d}a,
\]
which is smaller than the total density by a factor of order $\lambda^{1-\alpha}$. Balancing this slow renewal flux with jumps of size $\varepsilon$ selects the macroscopic time scale $\varepsilon^{-2/\alpha}$.

Two different examples are discussed after we introduce the general derivation framework. The first one is the 
age-structured kinetic models, that have previously been used to obtain fractional equations, sometimes formally \cite{EstradaPainter_2019,Nikos2025} and with methods avoiding Laplace transform in \cite{berry2026GLQ}. The second one is a pathway based renewal process, where the internal dynamics $a$ can be certain intracelluar protein level \cite{Perthame2018,Xue2021}, that is evolved by a stochastic differential equation \cite{tu2005white,XZZT2025}. Subdiffusion limit of the second type of model has never been discussed before, which provides the most common microscopic mechanism for the appearance of power-law residence time. 

The paper is organized as follows. In \cref{sec:general}, we formulate the general renewal--jump model, state the resolvent criterion, and prove the abstract subdiffusion limit. In \cref{sec:age}, we embed the age-structured model into this framework, compute its resolvent, derive the fractional equation, and discuss the associated renewal-time decay. In \cref{sec:pathway}, we apply the same framework to a degenerate pathway diffusion in the internal state. Regularity estimates for the limiting equation are recalled in \cref{sec:regularity}. Laplace-transform and Tauberian tools are collected in the appendix.

\section{A general renewal--jump framework}
\label{sec:general}

\subsection{The general dynamics}

Let $x\in\R^d$ denote physical position and let $a$ belong to an internal state space $E$. We write $\langle f\rangle_a$ for integration over $E$. The unknown $u_\eps(t,a,x)\ge0$
is the density of particles at time $t$, position $x$ and with internal state $a$.

At each renewal, the rate at which a particle located at $y$ jumps to a new position $x$ is governed by a transition probability $\omega_\epsilon(x,y)$, which satisfies the following conditions: it is a measurable function such that
\begin{equation}\label{as:Subomega}
 \omega_\epsilon(x,y) \geq 0, \qquad \int_{\mathbb{R}^d} \omega_\epsilon(x,y) \, \dd x = 1. 
\end{equation}

We consider the following microscopic equation
\begin{subequations}\label{gen:eq}
\begin{align} 
&\eps^\beta  \partial_t u_\eps (t,a,x) + \caL_a u_\eps +d(a) u_\eps=M(a) U_\eps(t,x), \label{gen:eq1}
\\
&U_\eps(t,x) = \iint  d(a) u_\eps(t,a,y) \omega_\eps(x,y) \dd y \dd a, \label{gen:eq2}
\end{align}
\end{subequations}
with the initial condition 
\begin{equation}
 u_\eps (0,a,x) = u^{\rm in}(a,x). 
\end{equation}
Here the internal state evolves according to a positivity preserving operator $\caL_a$  acting only on the variable~$a$ and which has a conservative form
\begin{align} 
\caL_a^* 1=0.
\end{align} 
The renewal or escape rate is $d(a)\ge0$. The post-renewal internal state is distributed according to a probability measure $M(a)$ on $E$.  The equation is understood weakly when $M$ is a measure and we assume
\begin{equation}
 \langle M(a)\rangle_a=1.
 \label{eq:conservative-internal}
\end{equation}
Then $U_\eps(t,x)$ is the flux of newly renewed, or reinjected, particles arriving at $x$.

The spatial density at location $x$ and the renewal flux of particles leaving~$x$ are
\begin{equation}
 \rho_\eps(t,x)=\langle u_\eps(t,\cdot,x)\rangle_a,
 \qquad
 R_\eps(t,x)=\langle d(a)\,u_\eps(t,\cdot,x)\rangle_a.
 \label{eq:rho-R}
\end{equation}
Integrating \eqref{gen:eq1} in $a$ gives the exact balance law
\begin{equation}
 \eps^\beta\partial_t\rho_\eps+R_\eps-U_\eps=0.
 \label{eq:mass-balance}
\end{equation}
Using \eqref{eq:mass-balance} and \eqref{as:Subomega} which implies $\int R_\eps(t,x)\dd x=\int U_\eps(t,x)\dd x$, we obtain conservation of total mass,
\begin{equation}
 \int_{\R^d}\rho_\eps(t,x)\,\dd x
 =\int_{\R^d}\rho^{\rm in}(x)\,\dd x.
 \label{eq:mass-conservation}
\end{equation}

Allowing $M(a)$ to be a probability measure is not a cosmetic generalization. The age-structured boundary condition is obtained by taking $E=\R$, $\caL_a=\partial_a$, and $M=\delta_0$. The age model is therefore an exact specialization of \eqref{gen:eq}; see \cref{sec:age-embedding}.

\paragraph{Assumption of the transition probability $\omega_\eps(x,y)$.}
 We assume that, additionally to~\eqref{as:Subomega}, the transition probability $\omega_\eps(x,y)$ satisfies an unbiased finite-variance condition. More precisely, as~$\eps \to 0$,
\begin{align} \label{as:Subomegaeps}
\begin{cases}
\dis \frac{1}{\eps^2}\int_{\R^d}(y-x)\omega_\eps(x,y)\,\dd x\longrightarrow0,\\[10pt]
\dis \frac{1}{\eps^2}  \int_{\R^d}  (x-y)\otimes(x-y) \omega_\eps(x,y) \dd x \to A(y) >0  \qquad \text{a symmetric positive matrix},
\\[10pt]
\dis \frac{1}{\eps^3} \int_{\R^d}  |x-y|^3 \omega_\epsilon (x,y)\dd x \in L^\infty(\R^d) \qquad  \text{ uniformly in } \eps.
\end{cases}
\end{align}
locally uniformly in $x$. The scaling assumptions in \eqref{as:Subomegaeps} encode the physical requirement that the microscopic jump process converges to a local diffusive behavior. The parameter $\eps$ represents the typical jump length; as $\eps \to 0$, jumps become infinitesimal. To observe a non-trivial macroscopic evolution over finite time scales, the jump rate must be amplified by a factor of $\eps^{-2}$. The first condition ensures the absence of a persistent macroscopic drift; The second condition guarantees that the mean squared displacement per unit time stabilizes to a finite, positive value; The third condition is a uniform integrability estimate that prevents rare, long-range jumps from dominating the dynamics.

The  jump generator and its dual are defined as
\begin{equation}\label{eq:jump-generator-adjoint}
\J_\eps u (x)
 :=\frac{1}{\eps^2}\int_{\R^d} u(y)\omega_\eps(x,y)\,\dd x - u(x), \qquad \J_\eps^*\varphi(y)
 :=\frac{1}{\eps^2}\int_{\R^d}
   \bigl[\varphi(x)-\varphi(y)\bigr]\omega_\eps(x,y)\,\dd x.
\end{equation}
The weak limit of \(\J_\eps u (x)\) is the Fokker--Planck operator 
\begin{equation}\label{eq:FK}
\J\rho=\sum_{i,j=1}^d\partial_{ij}\bigl(A_{ij}(x)\rho\bigr),
\end{equation}
because, for smooth test functions $\varphi$, Taylor expansions combined with \eqref{as:Subomegaeps} yields the convergence of \(\J_\eps^*\varphi\) to the second-order differential operators
\begin{equation}\label{eq:jump-generator-limit}
\J^*\varphi=A:D^2\varphi.
\end{equation}
Indeed,  the first-order term vanishes by unbiasedness and the higher-order terms are negligible due to the third-moment bound. 

In the standard homogeneous setting $\omega_\eps(x,y)=\eps^{-d}\omega \big((x-y)/\eps\big)$ with $\omega$ symmetric and integrable, these conditions are automatically satisfied and the diffusivity reduces to the constant matrix
\begin{equation}
A=\frac12\int_{\R^d}z\otimes z\,\omega(z)\,\dd z.
\end{equation}
For example, in one dimension the triangular kernel $\omega(z)=(1-|z|)_+$ yields $A=1/12$.

\paragraph{The internal resolvent and the heavy-tail assumption.}
For a scalar constant $\lambda>0$, let $Q_\lambda$ satisfy
\begin{equation}
 \lambda Q_\lambda+\caL_a Q_\lambda+d(a)Q_\lambda=M(a).
 \label{eq:Qlambda-general}
\end{equation}
Then, we define
\begin{equation}
 m(\lambda)=\langle Q_\lambda\rangle_a,
 \qquad
 r(\lambda)=\langle d(a)Q_\lambda\rangle_a.
 \label{eq:m-r}
\end{equation}
Integrating \eqref{eq:Qlambda-general} and using \eqref{eq:conservative-internal}, we find the exact identity
\begin{equation}
 \lambda m(\lambda)+r(\lambda)=1.
 \label{eq:resolvent-conservation}
\end{equation}
The subdiffusive regime is encoded in the following resolvent heavy-tail assumption.

\medskip
There exist $0<\alpha<1$ and $\gamma>0$ such that
\begin{equation}
 m(\lambda)=\gamma\lambda^{\alpha-1}\bigl(1+o(1)\bigr)
 \qquad\text{as }\lambda\downarrow0.
 \label{eq:resolvent-heavy-tail}
\end{equation}
Then \eqref{eq:resolvent-conservation} immediately yields
\begin{equation}
 r(\lambda)=1-\gamma\lambda^\alpha\bigl(1+o(1)\bigr).
 \label{eq:r-asymptotic}
\end{equation}
Formally, when $\lambda\to 0$, $Q_\lambda$ converges to a stationary profile $Q$ satisfying
\[
 (\caL_a+d(a))Q(a)=M(a),
 \qquad \langle d(a)Q(a)\rangle_a=1,
\]
but \eqref{eq:resolvent-heavy-tail} corresponds to $\langle Q\rangle_a=\infty$. This non-integrability is the microscopic origin of the slow macroscopic dynamics and represents the 'fat tail' distribution of the steady state of~\eqref{gen:eq}, i.e., its behavior as $\eps \to 0$.

\paragraph{Assumptions on the initial condition.}
To control the contribution of the initial internal distribution, set
\[
 G_\lambda=(\lambda+\caL_a+d)^{-1}.
\]
We assume, locally in $x$, that
\begin{equation}
 \bigl\langle G_\lambda u^{\rm in}(\cdot,x)\bigr\rangle_a
 \le C(x)m(\lambda),
 \qquad
 \bigl\langle d(a)G_\lambda u^{\rm in}(\cdot,x)\bigr\rangle_a\le C(x),
 \label{eq:initial-resolvent-bound}
\end{equation}
with $C\in L^1_{\rm loc}(\R^d)$. These assumptions permit non-equilibrium initial internal states but exclude an initial tail heavier than the stationary trapping tail. The latter case can retain additional  terms in the limit and is not treated here.

\subsection{Abstract subdiffusion limit}

The only available bound, uniform in $\eps$, for solutions of~\eqref{gen:eq} is \eqref{eq:mass-balance}. Therefore we consider measure valued solutions in the limit and establish the following convergence result.

\begin{theorem}[General renewal criterion for subdiffusion]
\label{thm:general-subdiffusion}
Assume \eqref{as:Subomega}, \eqref{as:Subomegaeps}, \eqref{eq:resolvent-heavy-tail}, \eqref{eq:conservative-internal}, and \eqref{eq:initial-resolvent-bound}. Choose
\begin{equation}
 \beta=\frac{2}{\alpha}.
 \label{eq:beta-scaling}
\end{equation}
Let $u_\eps$ be a nonnegative solution of~\eqref{gen:eq} with uniformly bounded total mass. Then every weak measure limit $\rho$ of a subsequence extracted from $\rho_\eps$ satisfies
\begin{equation}
 \CapD\rho-\frac{1}{\gamma}\J\rho=0,
 \qquad
 \rho(0,x)=\rho^{\rm in}(x),
 \label{eq:general-Caputo-limit}
\end{equation}
in the sense of distributions. Equivalently,
\begin{equation}
 \partial_t\int_0^t\frac{\rho(\tau,x)}{(t-\tau)^\alpha}\,\dd\tau
 -\frac{\Gamma(1-\alpha)}{\gamma}\J\rho(t,x)
 =t^{-\alpha}\rho^{\rm in}(x).
 \label{eq:general-RL-limit}
\end{equation}
If \eqref{eq:general-Caputo-limit} has a unique measure-valued solution, the full sequence converges.
\end{theorem}

\begin{remark}[The coefficient and normalization]
The coefficient in \eqref{eq:general-RL-limit} is $\Gamma(1-\alpha)/\gamma$. This is the reciprocal of the resolvent prefactor, up to the normalization of the fractional derivative. In the age-structured example, $\gamma=\Gamma(1-\alpha)$, so the unnormalized form \eqref{eq:general-RL-limit} has unit diffusion coefficient.
\end{remark}

\begin{remark}[Interpretation of the initial term]
The singular source $t^{-\alpha}\rho^{\rm in}$ in \eqref{eq:general-RL-limit} does not indicate that the solution jumps from zero at $t=0$. It is exactly the Riemann--Liouville representation of the standard Caputo initial condition $\rho(0)=\rho^{\rm in}$. Writing the equation in the Caputo form \eqref{eq:general-Caputo-limit} avoids this possible ambiguity.
\end{remark}

\begin{remark}[Laplace--weak formulation]
For measure-valued limits, \eqref{eq:general-Caputo-limit} may be taken to mean that, for every $s>0$ and $\varphi\in C_c^3(\R^d)$,
\begin{equation}
 s^\alpha\int_{\R^d}\widehat\rho(s,x)\varphi(x)\,\dd x
 -s^{\alpha-1}\int_{\R^d}\rho^{\rm in}(x)\varphi(x)\,\dd x
 -\frac{1}{\gamma}\int_{\R^d}\widehat\rho(s,x)\J^*\varphi(x)\,\dd x=0.
 \label{eq:Laplace-weak-definition}
\end{equation}
This formulation is also the one obtained directly from the microscopic equations and avoids imposing time regularity before it has been established.
\end{remark}

\noindent \textbf{Proof:} 
For $s>0$, let a hat denote the Laplace transform in $t$ and set
\[
 \lambda=\eps^\beta s.
\]
The Laplace transform of equation \eqref{gen:eq} becomes
\begin{equation}
 (\lambda+\caL_a+d(a))\widehat u_\eps
 =M(a)\widehat U_\eps+\eps^\beta u^{\rm in},
 \label{eq:general-Laplace}
\end{equation}
from which we infer that
\[
 \widehat u_\eps
 =Q_\lambda\widehat U_\eps+\eps^\beta G_\lambda u^{\rm in}.
 \label{eq:resolvent-decomposition}
\]
Integrating once with weight $1$ and once with weight $d(a)$ yields
\[ \widehat\rho_\eps
 =m(\lambda)\widehat U_\eps+\eps^\beta \langle G_\lambda u^{\rm in}\rangle_a,
\qquad
 \widehat R_\eps
 =r(\lambda)\widehat U_\eps+\eps^\beta\langle d(a)G_\lambda u^{\rm in}\rangle_a.
\]
Eliminating $\widehat U_\eps$ gives
\begin{equation}
 \widehat R_\eps
 =\frac{r(\lambda)}{m(\lambda)}\widehat\rho_\eps
 +\eps^\beta\left(\langle dG_\lambda u^{\rm in}\rangle_a-\frac{r(\lambda)}{m(\lambda)}\langle G_\lambda u^{\rm in}\rangle_a\right).
 \label{eq:R-rho-exact}
\end{equation}
By \eqref{eq:resolvent-heavy-tail}, \eqref{eq:r-asymptotic}, and \eqref{eq:initial-resolvent-bound}, we conclude that
\begin{equation}
 \widehat R_\eps
 =\frac{1}{\gamma}(\eps^\beta s)^{1-\alpha}\widehat\rho_\eps
 +o\bigl(\eps^{\beta(1-\alpha)}\bigr)
 \label{eq:key-flux-density}
\end{equation}
weakly on compact sets. This flux--density relation is the central closure. It replaces the equilibrium averaging used in normal diffusion.

Taking the Laplace transform of \eqref{eq:mass-balance}, testing against $\varphi\in C_c^3(\R^d)$, and using \eqref{eq:jump-generator-adjoint}, we obtain
\begin{equation}
 \eps^\beta\left[
 s\int_{\R^d}\widehat\rho_\eps\varphi\,\dd x
 -\int_{\R^d}\rho^{\rm in}\varphi\,\dd x\right]
 -\eps^2\int_{\R^d}\widehat R_\eps\J_\eps^*\varphi\,\dd x=0.
 \label{eq:weak-Laplace-balance}
\end{equation}
Divide by $\eps^\beta$ and multiply by $s^{\alpha-1}$. Substituting \eqref{eq:key-flux-density} gives
\begin{align}
 s^\alpha\int\widehat\rho_\eps\varphi
 -s^{\alpha-1}\int\rho^{\rm in}\varphi
 -\frac{1}{\gamma}\eps^{2-\alpha\beta}
 \int\widehat\rho_\eps\J_\eps^*\varphi=o(1).
 \label{eq:balance-scaled}
\end{align}
The choice $\alpha\beta=2$ makes all three terms of the same order. Passing to the limit using \eqref{eq:jump-generator-limit}, we find
\begin{equation}
 s^\alpha\int\widehat\rho\varphi
 -s^{\alpha-1}\int\rho^{\rm in}\varphi
 -\frac{1}{\gamma}\int\widehat\rho\J^*\varphi=0.
 \label{eq:limit-Laplace-weak}
\end{equation}
This is the Laplace transform of \eqref{eq:general-Caputo-limit}. The equivalent form \eqref{eq:general-RL-limit} follows from
\[
 \Lap_s\left[\partial_t\int_0^t\frac{f(\tau)}{(t-\tau)^\alpha}\,\dd\tau\right]
 =\Gamma(1-\alpha)s^\alpha\widehat f(s),
 \qquad
 \Lap_s[t^{-\alpha}]=\Gamma(1-\alpha)s^{\alpha-1}.
\]
This proves the theorem.

\begin{remark}
The same framework contains the normal diffusion limit. Suppose instead that
\begin{equation*}
m(0)=\lim_{\lambda\downarrow0}m(\lambda)=\overline m<\infty.
\end{equation*}
 Choosing $\beta=2$ in \eqref{eq:mass-balance} gives the Fokker-Planck equation (see \eqref{eq:FK})
\begin{equation*}
 \partial_t\rho-\frac{1}{\overline m}\J\rho=0,
 \qquad \rho(0)=\rho^{\rm in}.
\end{equation*}
\end{remark}

\section{The age-structured model}
\label{sec:age}

In our context, the age-structured equation with jumps is written
\begin{equation}
 \begin{cases}
 \eps^\beta\partial_tu_\eps+\partial_au_\eps+d(a)u_\eps=0,
 &a>0,\ t>0,\ x\in\R^d,\\[2mm]
 u_\eps(t,0,x)=U_\eps(t,x)
 =\displaystyle\int_{\R^d}\int_0^\infty
 d(a)u_\eps(t,a,y)\omega_\eps(x,y)\,\dd a\dd y,\\[2mm]
 u_\eps(0,a,x)=u^{\rm in}(a,x).
 \end{cases}
 \label{eq:age-model}
\end{equation}
It is a standard equation (see \cite{PerthameTEB} and references therein) which arises to describe trapping times between spatial jumps and has been already studied in particular in the regime of interest here \cite{BerryLG2016, CGA_2019}.

We first show how to embed this example in our general framework of  Section~\ref{sec:general} and establish its subdiffusion limit. We conclude with a remark on the long term convergence for the simple age structure equation without space.

\subsection{Embedding the renewal boundary into the abstract equation}
\label{sec:age-embedding}

Extend $u_\eps$ by zero to $a<0$. In the distributional sense on $\R_a$,
\[
 \partial_a\bigl(\mathbf 1_{a>0}u_\eps\bigr)
 =\mathbf 1_{a>0}\partial_au_\eps+u_\eps(t,0,x)\delta_0(a).
\]
Consequently, \eqref{eq:age-model} is equivalent to
\begin{equation}
 \eps^\beta\partial_tu_\eps+\partial_au_\eps+d(a)u_\eps
 =\delta_0(a)U_\eps(t,x).
 \label{eq:age-abstract-form}
\end{equation}
It is therefore exactly \eqref{gen:eq} with
\[
 E=\R_+,
 \qquad
 \caL_a=\partial_a,
 \qquad
 M(a)=\delta_0.
 \label{eq:age-identification}
\]
This identification is the reason the age-structured derivation should be organized through the general resolvent theorem rather than repeated independently.

The internal resolvent equation is
\[
 \lambda Q_\lambda+\partial_aQ_\lambda+d(a)Q_\lambda=\delta_0,
 \qquad Q_\lambda(a)=0\quad(a<0),
 \label{eq:age-resolvent-distribution}
\]
which is equivalent to the boundary-value problem
\begin{equation}
 \lambda Q_\lambda+\partial_aQ_\lambda+d(a)Q_\lambda=0,
 \quad a>0,
 \qquad Q_\lambda(0)=1.
 \label{eq:age-resolvent-bvp}
\end{equation}

Set
$D(a)=\int_0^a d(a')\,\dd a'$. Hence
\[
 Q_\lambda(a)=\mathbf 1_{a\ge0}\exp\bigl[-D(a)-\lambda a\bigr].
 \label{eq:age-Qlambda}
\]
All information needed by the general \cref{thm:general-subdiffusion} is therefore contained in the single scalar function
\begin{equation*}
 m(\lambda)=\int_0^\infty e^{-D(a)-\lambda a}\,\dd a.
\end{equation*}

\subsection{Power-law escape rate and resolvent asymptotics}

We now choose
\begin{equation}
 d(a)=\frac{\alpha}{1+a},
 \qquad 0<\alpha<1.
 \label{eq:power-escape}
\end{equation}
Then
\begin{equation}
 D(a)=\alpha\log(1+a),
 \qquad e^{-D(a)}=(1+a)^{-\alpha}.
 \label{eq:power-equilibrium}
\end{equation}
The stationary profile is $e^{-D(a)}$ which is not integrable, while $d(a)e^{-D(a)}$ is integrable and has mass one. This is the infinite-mean residence-time regime.

\begin{lemma}[Age resolvent asymptotics]
\label{lem:age-resolvent}
For the age structures equation \eqref{eq:age-resolvent-bvp} with $d(a)$ determined by \eqref{eq:power-escape}, we have 
\begin{equation}
 m(\lambda)
 =\int_0^\infty\frac{e^{-\lambda a}}{(1+a)^\alpha}\,\dd a
 \sim\Gamma(1-\alpha)\lambda^{\alpha-1}
 \qquad(\lambda\downarrow0).
 \label{eq:age-m-asymptotic}
\end{equation}
Moreover,
\begin{equation}
 r(\lambda)=\int_0^\infty d(a)Q_\lambda(a)\,\dd a
 =1-\lambda m(\lambda)
 =1-\Gamma(1-\alpha)\lambda^\alpha(1+o(1)).
 \label{eq:age-r-asymptotic}
\end{equation}
\end{lemma}

\begin{proof}
With $z=\lambda a$,
\[
 m(\lambda)
 =\lambda^{\alpha-1}\int_0^\infty
 \frac{e^{-z}}{(\lambda+z)^\alpha}\,\dd z.
\]
The integral converges to $\int_0^\infty e^{-z}z^{-\alpha}\,\dd z=\Gamma(1-\alpha)$ by a standard splitting near $z=0$ and dominated convergence away from zero. The identity for $r$ follows either from \eqref{eq:resolvent-conservation} or by integrating \eqref{eq:age-resolvent-bvp}.
\end{proof}

The initial resolvent bounds also follow directly by straight forward explicit calculations. For $f\ge0$, 
\begin{equation}
 (G_\lambda f)(a)
 =\int_0^a e^{-D(a)+D(a')-\lambda(a-a')}f(a')\,\dd a'.
 \label{eq:age-G}
\end{equation}
If we assume that 
\begin{equation}
 \int_{\R^d}\int_0^\infty(1+a)^\alpha u^{\rm in}(a,x)\,\dd a\dd x<\infty,
 \label{eq:age-initial-weight}
\end{equation}
then, by Fubini's theorem,
\begin{align}
 \langle G_\lambda u^{\rm in}\rangle_a
 &\le m(\lambda)\int_0^\infty(1+a)^\alpha u^{\rm in}(a,x)\,\dd a,\label{eq:age-H0-bound}
 \\
 \langle d(a)G_\lambda u^{\rm in}\rangle_a
 &\le \int_0^\infty u^{\rm in}(a,x)\,\dd a.
 \label{eq:age-H1-bound}
\end{align}
 Thus \eqref{eq:initial-resolvent-bound} is verified.

The abstract theorem can now be applied without repeating the multiscale calculation.

\begin{theorem}[Subdiffusion limit of the age-structured jump model]
\label{thm:age-limit}
Assume the spatial jump condition \eqref{eq:jump-generator-limit}, the escape rate \eqref{eq:power-escape}, nonnegative initial data with finite mass, and \eqref{eq:age-initial-weight}. Let
\begin{equation}
 \beta=\frac{2}{\alpha}.
\end{equation}
Then every weak limit of
\[
 \rho_\eps(t,x)=\int_0^\infty u_\eps(t,a,x)\,\dd a
\]
solves
\begin{equation}
 \CapD\rho-\frac{1}{\Gamma(1-\alpha)}\J\rho=0,
 \qquad \rho(0,x)=\rho^{\rm in}(x).
 \label{eq:age-Caputo-limit}
\end{equation}
Equivalently, we may write
\begin{equation}
 \partial_t\int_0^t\frac{\rho(\tau,x)}{(t-\tau)^\alpha}\,\dd\tau
 -\J\rho(t,x)
 =t^{-\alpha}\rho^{\rm in}(x).
 \label{eq:age-RL-limit}
\end{equation}
where $\J$ takes the Fokker-Plank form $\J\rho=\sum_{i,j}\partial_{ij}(A_{ij}\rho)$ as in \eqref{eq:FK}.
\end{theorem}

\begin{proof}
By \cref{sec:age-embedding}, the age structured equation is the abstract model with $M=\delta_0$. Lemma \ref{lem:age-resolvent} verifies \eqref{eq:resolvent-heavy-tail} with $\gamma=\Gamma(1-\alpha)$, while \eqref{eq:age-H0-bound}--\eqref{eq:age-H1-bound} verify the initial resolvent assumptions. The result follows from \cref{thm:general-subdiffusion}.
\end{proof}

\begin{remark}[Dependence on the initial age distribution]
Under \eqref{eq:age-initial-weight}, the detailed initial age profile disappears from the leading macroscopic equation; only its spatial mass $\rho^{\rm in}(x)$ remains. This is a conclusion of the resolvent estimate, not an assumption that particles are initially at equilibrium. If the initial population contains an age tail comparable to or heavier than $(1+a)^{-\alpha}$ without the weighted integrability in \eqref{eq:age-initial-weight}, an additional aging contribution may survive and the right-hand side of \eqref{eq:age-RL-limit} need not reduce to $t^{-\alpha}\rho^{\rm in}$.
\end{remark}

\begin{remark}The same framework contains the diffusion limit.
For a constant escape rate $d(a)=D_0$, one has $Q(a)=e^{-D_0a}$ and $\overline m=1/D_0$, so the effective diffusion coefficient is $D_0$. For $d(a)=\alpha/(1+a)$ with $\alpha>1$, one has $\overline m=1/(\alpha-1)$ and the coefficient is $\alpha-1$.
\end{remark}

\subsection{Homogeneous renewal equation and decay of the renewal flux}
\label{sec:age-decay}

The same resolvent exponent also controls the large-time decay of the renewal flux. Consider
\begin{equation}
 \begin{cases}
 \partial_tu+\partial_au+d(a)u=0, &a>0,\ t>0,\\
 U(t)=u(t,0)=\displaystyle\int_0^\infty d(a)u(t,a)\,\dd a,\\
 u(0,a)=u^{\rm in}(a),
 \end{cases}
 \label{eq:homogeneous-age}
\end{equation}
with $u^{\rm in}\ge0$ and $\rho=\int_0^\infty u^{\rm in}(a)\,\dd a<\infty$. To see the connection with the resolvent framework directly, take the Laplace transform in time. The transformed characteristic equation gives
\[
 \widehat u(s,a)=\widehat U(s)e^{-D(a)-sa}
 +e^{-D(a)-sa}\int_0^a u^{\rm in}(a')e^{D(a')+sa'}\,\dd a'.
\]
Substitution into $\widehat U=\int_0^\infty d(a)\widehat u(s,a)\dd a$ and one integration by parts yield
\[
 s\,m(s)\widehat U(s)
 =\rho-s\int_0^\infty e^{-st}
 \left[\int_0^\infty u^{\rm in}(a)
 \left(\frac{1+a}{1+a+t}\right)^\alpha\,\dd a\right]\dd t.
\]
Inverting the Laplace transform gives the exact identity
\begin{equation}
 \int_0^t\frac{U(\tau)}{(1+t-\tau)^\alpha}\,\dd\tau
 =\rho-\int_0^\infty u^{\rm in}(a)
 \left(\frac{1+a}{1+a+t}\right)^\alpha\,\dd a.
 \label{eq:exact-renewal-identity}
\end{equation}
Both sides vanish at $t=0$. As $t\to\infty$, the second term on the right converges to zero, and therefore
\begin{equation}
 \lim_{t\to\infty}
 \int_0^t\frac{U(\tau)}{(1+t-\tau)^\alpha}\,\dd\tau=\rho.
 \label{eq:renewal-convolution-limit}
\end{equation}
In Laplace variables this implies
\begin{equation}
 \widehat U(s)\sim\frac{\rho}{\Gamma(1-\alpha)}s^{-\alpha}
 \qquad(s\downarrow0).
 \label{eq:U-Laplace-asymptotic}
\end{equation}
Thus $U$ has the characteristic decay scale $t^{\alpha-1}$. Under an additional standard Tauberian regularity assumption, for example eventual monotonicity,
\begin{equation}
 U(t)\sim\frac{\rho}{\Gamma(\alpha)\Gamma(1-\alpha)}t^{\alpha-1}.
 \label{eq:U-pointwise-asymptotic}
\end{equation}
Without such a pointwise assumption, the integrated bounds in \cref{app:tauberian} provide the appropriate rigorous interpretation. The important structural point is that the exponent in \eqref{eq:U-Laplace-asymptotic} is the same exponent that appears in the closure \eqref{eq:key-flux-density} and in the fractional macroscopic equation.

\section{A pathway-based internal-state model}
\label{sec:pathway}

The general framework is not restricted to residence time as the internal variable. We now consider an internal pathway state $a\in\R$ that diffuses in a confining but degenerate landscape. 

\subsection{The model}
Let $Q_0(a)>0$, $D(a)>0$, and $d(a)>0$, and define
\begin{equation}
 \caL_a u_\eps=-\partial_a\left(D(a)Q_0(a)\partial_a\frac{u_\eps}{Q_0(a)}\right),
 \qquad
 M(a)=d(a)Q_0(a).
 \label{eq:pathway-operator}
\end{equation}
We impose
\begin{equation}
 \int_\R d(a)Q_0(a)\,\dd a=1.
 \label{eq:pathway-normalization}
\end{equation}
Then, we have $\caL_a Q_0=0$ and 
\[
 \caL_a Q_0+d(a)Q_0=M(a).
\]
Thus $Q_0$ is the stationary internal profile. We assume that it is not integrable.

More precisely, for some $0<\sigma<1$, $k>0$, $n\ge0$, and positive constants $c_q,c_d,c_D$, assume that for~$|a|\gg 1$
\begin{align}
 Q_0(a)=c_q|a|^{-\sigma}(1+o(1)),\qquad
 d(a)=c_d|a|^{-k}(1+o(1)),\qquad
D(a)=c_D|a|^n(1+o(1)),\label{eq:pathwaytail}\end{align}
with corresponding derivative bounds for $|a|>a_0$. All functions $Q_0(a)$, $d(a)$ and $D(a)$ are continuous and bounded for $|a|<a_0$. 

The normalization \eqref{eq:pathway-normalization} requires
\begin{equation}
 \sigma+k>1.
 \label{eq:pathway-integrability}
\end{equation}
The resolvent profile solves
\begin{equation}
 \lambda Q_\lambda
 -\partial_a\left(DQ_0\partial_a\frac{Q_\lambda}{Q_0}\right)
 +d(a)Q_\lambda=d(a)Q_0.
 \label{eq:pathway-resolvent}
\end{equation}
In order to estimates the parameter $\alpha$, i.e., the mass $m(\lambda)$ (see \cref{prop:pathway-resolvent}) we need some preparation.

\begin{lemma}[Energy estimate]
\label{lem:pathway-energy}
Set $m(\lambda)=\int_\R Q_\lambda\,\dd a$. Assume that the boundary flux vanishes at infinity. Then
\begin{align}
&\lambda\int_\R\frac{Q_\lambda^2}{Q_0}\,\dd a
 +\int_\R DQ_0\left|\partial_a\frac{Q_\lambda}{Q_0}\right|^2\,\dd a
 +\int_\R dQ_0\left(\frac{Q_\lambda}{Q_0}-1\right)^2\,\dd a
 =\lambda m(\lambda),\label{eq:pathway-energy-identity}
\\
 &\int_\R d(a)Q_\lambda(a)\,\dd a=1-\lambda m(\lambda).
\label{eq:pathway-r-identity}
\end{align}
In particular,
\begin{align}
 \int_\R\frac{Q_\lambda^2}{Q_0}\,\dd a\le m(\lambda),\qquad
 \int_\R DQ_0\left|\partial_a\frac{Q_\lambda}{Q_0}\right|^2\,\dd a\le\lambda m(\lambda),\qquad
 \int_\R dQ_0\left(\frac{Q_\lambda}{Q_0}-1\right)^2\,\dd a&\le\lambda m(\lambda).\label{eq:pathway-energyinequality}
\end{align}
\end{lemma}
\begin{proof}
On the one hand, we multiply \eqref{eq:pathway-resolvent} by $Q_\lambda/Q_0$,  integrate with respect to $a$ over $\R$ and integrate my parts the second order term. This gives
\begin{equation}
\lambda\int\frac{Q_\lambda^2}{Q_0}\dd a
 +\int DQ_0\left|\partial_a\frac{Q_\lambda}{Q_0}\right|^2\dd a
 +\int d(a)Q_0\left(\frac{Q_\lambda}{Q_0}\right)^2\dd a
 =\int d(a)Q_\lambda\dd a.\label{eq:pathway-energy}
\end{equation}
On the other hand, integrating \eqref{eq:pathway-resolvent} directly
and using the vanishing boundary flux and
$\int_{\mathbb R}d(a)Q_0\dd a=1$, we obtain
\[
\lambda m(\lambda)+\int_{\mathbb R}dQ_\lambda\,da=1,
\]
that gives \eqref{eq:pathway-r-identity}.

Moreover, we may write 
\[
\begin{aligned}
\int_{\mathbb R}
d(a)Q_0
\left(\frac{Q_\lambda}{Q_0}\right)^2\,da=
\int_{\mathbb R}
d(a)Q_0
\left(
\frac{Q_\lambda}{Q_0}-1
\right)^2\,da
+2\int_{\mathbb R}
d(a)Q_0
\left(
\frac{Q_\lambda}{Q_0}-1
\right)\,da
+\int_{\mathbb R}d(a)Q_0\,da.
\end{aligned}
\]
By \eqref{eq:pathway-r-identity}, we conclude
\[
\int_{\mathbb R}
d(a)Q_0
\left(
\frac{Q_\lambda}{Q_0}-1
\right)\,da
=\int_{\mathbb R}d(a)Q_\lambda\,da-1
=-\lambda m(\lambda).
\]
Consequently, it holds
\[
\int_{\mathbb R}
dQ_0
\left(\frac{Q_\lambda}{Q_0}\right)^2\,da
=
\int_{\mathbb R}
dQ_0
\left(
\frac{Q_\lambda}{Q_0}-1
\right)^2\,da
-2\lambda m(\lambda)+1.
\]
Substituting this and \eqref{eq:pathway-r-identity} into \eqref{eq:pathway-energy} yields \eqref{eq:pathway-energy-identity}.
Since all three terms on the left-hand side of \eqref{eq:pathway-energy-identity} are nonnegative,
we immediately obtain \eqref{eq:pathway-energyinequality}.
\end{proof}

\subsection{Resolvent tail and anomalous exponent}

Rearranging \eqref{eq:pathway-resolvent}, we obtain
\[
 Q_\lambda
 =\frac{d(a)}{d(a)+\lambda}Q_0
 +\frac{1}{d(a)+\lambda}\partial_a
 \left(DQ_0\partial_a\frac{Q_\lambda}{Q_0}\right).
\]
After integration, we have
\begin{equation}
 m(\lambda)=J_1(\lambda)+J_2(\lambda),
 \label{eq:pathway-J12}
\end{equation}
where
\begin{align}
J_1(\lambda)=\int_\R\frac{d(a)}{d(a)+\lambda}Q_0(a)\,\dd a,
 \qquad
J_2(\lambda)&=\int_\R\frac{1}{d(a)+\lambda}\partial_a
\left(DQ_0\partial_a\frac{Q_\lambda}{Q_0}\right)\,\dd a.
 \label{eq:pathway-J}
\end{align}

\paragraph{The first term $J_1(\lambda)$.}
For $J_1(\lambda)$, we split the integral into a bounded
region and a tail region. Let $a_0>0$ be sufficiently large such
that the asymptotic assumptions in \eqref{eq:pathwaytail}
hold for $|a|\geq a_0$. Then
\[
\begin{aligned}
J_1(\lambda)
&=
\int_{|a|\leq a_0}
\frac{d(a)Q_0(a)}{d(a)+\lambda}\,da
+
\int_{|a|>a_0}
\frac{d(a)Q_0(a)}{d(a)+\lambda}\,da
\\
&=
O(1)
+
\int_{|a|>a_0}
\frac{d(a)Q_0(a)}{d(a)+\lambda}\,da.
\end{aligned}
\]
For $|a|>a_0$, the asymptotic assumptions 
$Q_0(a)\sim c_q|a|^{-\sigma}$,
$d(a)\sim c_d|a|^{-k}$, and that the leading-order integrand is even, give
\[
\begin{aligned}
\int_{|a|>a_0}
\frac{d(a)Q_0(a)}{d(a)+\lambda}\,da
\sim
c_dc_q
\int_{|a|>a_0}
\frac{|a|^{-k-\sigma}}
{|a|^{-k}+\lambda}\,da=
c_dc_q
\int_{|a|>a_0}
\frac{|a|^{-\sigma}}
{1+\lambda|a|^k}\,da.
=
2c_dc_q
\int_{a_0}^{\infty}
\frac{a^{-\sigma}}
{1+\lambda a^k}\,da.
\end{aligned}
\]
Introducing the rescaled variable
\[
z=\lambda^{1/k}a,
\qquad
a=\lambda^{-1/k}z,
\qquad
\dd a=\lambda^{-1/k}\dd z,
\]
we obtain
\[
\begin{aligned}
J_1(\lambda)
\sim
2c_dc_q
\int_{\lambda^{1/k}a_0}^{\infty}
\frac{
\lambda^{\sigma/k}z^{-\sigma}
}{
1+z^k
}
\lambda^{-1/k}\dd z=
2c_dc_q
\lambda^{(\sigma-1)/k}
\int_{\lambda^{1/k}a_0}^{\infty}
\frac{z^{-\sigma}}{1+z^k}\dd z.
\end{aligned}
\]
Since $0<\sigma<1$ and $k>0$,
\[
\frac{z^{-\sigma}}{1+z^k}\in L^1(0,\infty),
\]
and $\lambda^{1/k}a_0\to0$ as $\lambda\to0$. Hence,
by the dominated convergence theorem,
\[
\begin{aligned}
J_1(\lambda)
&\sim
2c_dc_q
\lambda^{-(1-\sigma)/k}
\int_0^\infty
\frac{z^{-\sigma}}{1+z^k}\dd z.
\end{aligned}
\]

The first term determines the leading order. Indeed, if $k>1-\sigma$, then
\begin{equation}
 J_1(\lambda)\sim\gamma_0\lambda^{-(1-\sigma)/k},
 \label{eq:J1-asymptotic}
\end{equation}
with $
 \gamma_0
 =2c_qc_d
 \int_0^\infty\frac{z^{-\sigma}}{1+z^k}\,\dd z
$.
The integral is finite precisely because $0<1-\sigma<k$.

\paragraph{The second term $J_2(\lambda)$.}
For the second term, integration by parts gives
\[
J_2(\lambda)
=
-\int_{\mathbb R}
DQ_0
\partial_a\frac{Q_\lambda}{Q_0}
\partial_a\left(\frac{1}{d(a)+\lambda}\right)\dd a.
\]
Since
\[
\partial_a\left(\frac{1}{d(a)+\lambda}\right)
=
-\frac{d'(a)}{(d(a)+\lambda)^2},
\]
the Cauchy--Schwarz inequality yields
\[
\begin{aligned}
|J_2(\lambda)|
&\leq
\left(
\int_{\mathbb R}
DQ_0
\left|
\partial_a\frac{Q_\lambda}{Q_0}
\right|^2\,da
\right)^{1/2}
\left(
\int_{\mathbb R}
DQ_0
\frac{|d'(a)|^2}{(d(a)+\lambda)^4}\dd a
\right)^{1/2}
\\
&\leq
\lambda^{1/2}m(\lambda)^{1/2}
\left(
\int_{\mathbb R}
DQ_0
\frac{|d'(a)|^2}{(d(a)+\lambda)^4}\dd a
\right)^{1/2},
\end{aligned}
\]
where we used \eqref{eq:pathway-energyinequality}.

For $|a|$ sufficiently large, the assumptions in \eqref{eq:pathwaytail}
and the corresponding derivative bounds imply
\[
DQ_0\frac{|d'(a)|^2}{(d(a)+\lambda)^4}
\leq
C\frac{|a|^{n-\sigma-2k-2}}
{(|a|^{-k}+\lambda)^4}.
\]
Using the change of variables
$a=\lambda^{-1/k} z$ in the tail region gives
\[
\begin{aligned}
\int_{|a|\geq a_0}
DQ_0\frac{|d'(a)|^2}{(d(a)+\lambda)^4}\dd a
&\leq
C\int_{|a|\geq a_0}
\frac{|a|^{n-\sigma-2k-2}}
{(|a|^{-k}+\lambda)^4}\dd a\leq
C\lambda^{\frac{1+\sigma-n-2k}{k}}
\int_{\mathbb R}
\frac{|z|^{n-\sigma-2k-2}}
{(1+|z|^{-k})^4}\dd z.
\end{aligned}
\]
Hence, whenever the last integral is finite,
\[
\int_{|a|\geq a_0}
DQ_0\frac{|d'|^2}{(d+\lambda)^4}\,da
\leq
C\lambda^{\frac{1+\sigma-n-2k}{k}}.
\]
On the compact region $|a|\leq a_0$, the coefficients are bounded and
$d+\lambda$ is bounded from below, so
\[
\int_{\mathbb R}
DQ_0\frac{|d'(a)|^2}{(d(a)+\lambda)^4}\dd a
\leq
C\lambda^{\frac{1+\sigma-n-2k}{k}}+C,
\]
and consequently
\begin{equation}\label{eq:J2-estimate}
|J_2(\lambda)|
\leq
C\lambda^{\frac{1+\sigma-n-2k}{2k}}
m(\lambda)^{1/2}
+
C\lambda^{1/2}m(\lambda)^{1/2}.
\end{equation}

\paragraph{Close the estimate.} Denote
\[
l:=\frac{1+\sigma-n-2k}{2k},
\]we use \eqref{eq:pathway-J12} 
together with \eqref{eq:J2-estimate}, let $C$ denote a arbitrary constant, for sufficiently small $\lambda$,
\[
m(\lambda)
\leq J_1(\lambda)
+C\lambda^l m(\lambda)^{1/2}
+C\lambda^{1/2}m(\lambda)^{1/2}.
\]
By Young's inequality, for any $\eta>0$,
\[
C\lambda^l m(\lambda)^{1/2}
\leq
\eta m(\lambda)
+\frac{C^2}{\eta}\lambda^{2l},
\qquad
C\lambda^{1/2}m(\lambda)^{1/2}
\leq
\eta m(\lambda)+\frac{C^2}{\eta}\lambda.
\]
Choosing $\eta>0$ sufficiently small and absorbing the resulting
terms into the left-hand side gives
\[m(\lambda)
\leq
C J_1(\lambda)
+C\lambda^{2l}
+C\lambda .\]
Since $J_1(\lambda)\sim\gamma_0\lambda^{-(1-\sigma)/k}$,
it remains to compare the two remainder terms with $J_1$.
In particular,
\[
\frac{\lambda^{2l}}{J_1(\lambda)}
=
O\left(
\lambda^{\frac{1+\sigma-n-2k}{k}
+\frac{1-\sigma}{k}}
\right)
=
O\left(
\lambda^{\frac{2-n-2k}{k}}
\right),
\]
which tends to zero provided $n+2k<2$.
Moreover,
\[
\frac{\lambda}{J_1(\lambda)}
=
O\left(
\lambda^{1+(1-\sigma)/k}
\right)
\longrightarrow 0.
\]
Consequently, $
m(\lambda)=O(J_1(\lambda))=
O\left(
\lambda^{-(1-\sigma)/k}
\right)$.
Substituting this estimate back into \eqref{eq:J2-estimate}, we obtain
\[
|J_2(\lambda)|
\leq
C\lambda^l
\lambda^{-\frac{1-\sigma}{2k}}
+
C\lambda^{1/2}
\lambda^{-\frac{1-\sigma}{2k}}=
C\lambda^{\frac{2\sigma-n-2k}{2k}}
+
C\lambda^{\frac{k-1+\sigma}{2k}}.
\]
Under the above conditions, both terms are
$o\left(\lambda^{-(1-\sigma)/k}
\right)=o(J_1(\lambda))$. Therefore,
\[
m(\lambda)
=
J_1(\lambda)+o(J_1(\lambda))
\sim J_1(\lambda)
\sim
\gamma_0\lambda^{-(1-\sigma)/k}.
\sim
\gamma_0\lambda^{\alpha-1}.
\]
with
$\alpha=1-\frac{1-\sigma}{k}$.

\begin{proposition}[Pathway resolvent asymptotics]
\label{prop:pathway-resolvent}
Assume \eqref{eq:pathwaytail}, $0<\sigma<1$, $k>0$, $k+\sigma>1$;  $n\ge0$, $n+2k<2$, and the regularity needed for \eqref{eq:J2-estimate}. Define
\begin{equation}
 \alpha=1-\frac{1-\sigma}{k}\in(0,1).
 \label{eq:pathway-alpha}
\end{equation}
Then
\begin{equation}
 m(\lambda)\sim\gamma_0\lambda^{\alpha-1},
 \qquad
 \int dQ_\lambda=1-\gamma_0\lambda^\alpha(1+o(1)).
 \label{eq:pathway-resolvent-asymptotic}
\end{equation}
\end{proposition}

Thus the pathway model satisfies exactly the same abstract criterion as the age model, although its microscopic internal dynamics is entirely different.

\begin{corollary}[Subdiffusion generated by pathway dynamics]
\label{cor:pathway-limit}
Under the assumptions of \cref{prop:pathway-resolvent}, the spatial jump assumptions of \cref{sec:general}, and the initial resolvent bounds \eqref{eq:initial-resolvent-bound}, choose $\beta=2/\alpha$. Then the spatial density converges, along weakly convergent subsequences, to
\begin{equation}
 \CapD\rho-\frac{1}{\gamma_0}\J\rho=0,
 \qquad \rho(0)=\rho^{\rm in}.
 \label{eq:pathway-limit}
\end{equation}
Equivalently,
\[
 \partial_t\int_0^t\frac{\rho(\tau)}{(t-\tau)^\alpha}\,\dd\tau
 -\frac{\Gamma(1-\alpha)}{\gamma_0}\J\rho
 =t^{-\alpha}\rho^{\rm in}.
\]
\end{corollary}

\section{Regularity of the limiting equation}
\label{sec:regularity}

The compactness supplied by mass conservation alone gives a measure-valued limit. Additional regularity follows from the fractional diffusion equation. We illustrate the energy structure in the age-normalized case $A=I$, for which
\begin{equation}
 \partial_t\int_0^t\frac{\rho(\tau,x)}{(t-\tau)^\alpha}\,\dd\tau
 -\Delta\rho(t,x)=t^{-\alpha}\rho^{\rm in}(x).
 \label{eq:regularity-equation}
\end{equation}
Let $v(t,x)=\rho(t,x)-\rho^{\rm in}(x)$. Then $v(0)=0$ and
\begin{equation}
 \partial_t\int_0^t\frac{v(\tau,x)}{(t-\tau)^\alpha}\,\dd\tau
 -\Delta v=\Delta\rho^{\rm in}.
 \label{eq:v-equation}
\end{equation}
For a scalar $v\in C^1([0,T])$ with $v(0)=0$, one has the fractional chain identity
\begin{align}
 v(t)\partial_t\int_0^t\frac{v(\tau)}{(t-\tau)^\alpha}\,\dd\tau
 &=\frac12\partial_t\int_0^t\frac{v(\tau)^2}{(t-\tau)^\alpha}\,\dd\tau
 +\frac{v(t)^2}{2t^\alpha}
 \nonumber\\
 &\quad+\frac{\alpha}{2}\int_0^t
 \frac{|v(\tau)-v(t)|^2}{(t-\tau)^{\alpha+1}}\,\dd\tau.
 \label{eq:fractional-chain}
\end{align}
Multiplying \eqref{eq:v-equation} by $v(t,x)$ and integrating in $x$ gives
\begin{align}
 &\frac12\partial_t\int_0^t\int_{\R^d}
 \frac{v(\tau,x)^2}{(t-\tau)^\alpha}\,\dd x\dd\tau
 +\int_{\R^d}|\nabla v(t,x)|^2\,\dd x
 +\frac{1}{2t^\alpha}\int_{\R^d}v(t,x)^2\,\dd x
 \nonumber\\
 &\qquad
 +\frac{\alpha}{2}\int_0^t\int_{\R^d}
 \frac{|v(\tau,x)-v(t,x)|^2}{(t-\tau)^{\alpha+1}}\,\dd x\dd\tau
 =-\int_{\R^d}\nabla v(t,x)\cdot\nabla\rho^{\rm in}(x)\,\dd x.
 \label{eq:energy-identity}
\end{align}
This gives an $H^1$ estimate whenever $\rho^{\rm in}\in H^1$. More generally, for every convex $\Phi$ with $\Phi(0)=0$,
\begin{equation}
 \partial_t\int_0^t\frac{\Phi(v(\tau))}{(t-\tau)^\alpha}\,\dd\tau
 \le \Phi'(v(t))\partial_t\int_0^t\frac{v(\tau)}{(t-\tau)^\alpha}\,\dd\tau.
 \label{eq:convex-chain}
\end{equation}
Such inequalities are standard starting points for $L^p$ and Sobolev estimates; see, for example, \cite{DongKim_2020,Allen2017,zacher_2012,FEB_2024}.

At the microscopic level, time regularity requires a preparation condition because the macroscopic time is accelerated by $\eps^{-\beta}$. In the age model, if
\begin{equation}
 |\partial_au^{\rm in}+d(a)u^{\rm in}|
 \le C\eps^\beta(1+a)^{-\alpha},
 \label{eq:well-prepared-time}
\end{equation}
then $z_\eps=\partial_tu_\eps$ satisfies the same positivity-preserving evolution and comparison gives a uniform bound on $\int d(a)|z_\eps|\,\dd a$. Consequently, the renewal flux has a uniform time-Lipschitz estimate. Obtaining estimates of this type without a near-equilibrium preparation remains an interesting question.

\section{Discussion}

Our general derivation of the subdiffusion equation can be summarized in three steps.
\begin{enumerate}[label=(\roman*)]
 \item The internal dynamics is compressed into the resolvent mass $m(\lambda)$. An integrable equilibrium gives $m(0)<\infty$ and normal diffusion; a divergent law $m(\lambda)\sim\gamma\lambda^{\alpha-1}$ gives subdiffusion.
 \item The resolvent converts total density into renewal flux through the relation
 \[
  \widehat R_\eps\sim\gamma^{-1}(\eps^\beta s)^{1-\alpha}\widehat\rho_\eps.
 \]
 This is the universal closure relation.
 \item Balancing the renewal flux with jumps of size $\eps$ selects $\alpha\beta=2$ and produces the Caputo equation.
\end{enumerate}

This organization makes the role of the age-structured model transparent. The age calculation is not a separate route to the limit; it is an explicit verification of the abstract resolvent criterion with $M=\delta_0$. The pathway model then demonstrates why the abstraction is useful: the same fractional equation follows from a qualitatively different internal mechanism once its resolvent has the same small-frequency singularity.

Several extensions are natural. Biased jump kernels add drift to $\J$; heavy-tailed jump lengths can combine time subdiffusion with spatial fractional operators; nonlinear renewal or interaction rates lead to nonlinear fractional equations; and non-prepared initial internal distributions can generate aging terms beyond the standard Caputo initial condition. The present framework isolates the part of the argument that remains unchanged across these variants.

We conclude by noticing that another route to subdiffusions is the theory of subordinators \cite{Baeumer2009}. 

\appendix

\section{Laplace-transform identities and integrated Tauberian bounds}
\label{app:tauberian}

For a function $f(t)$ supported in $t\ge0$, write
\[
 \widehat f(s)=\Lap_sf=\int_0^\infty e^{-st}f(t)\,\dd t.
\]
If $f$ has no atom at $t=0$, then $\Lap_s(\partial_tf)=s\widehat f$. For $0<\alpha<1$,
\begin{align}
 \Lap_s(t^{-\alpha})&=\Gamma(1-\alpha)s^{\alpha-1},\label{eq:Lap-power}\\
 \Lap_s\left(\partial_t\int_0^t\frac{f(\tau)}{(t-\tau)^\alpha}\,\dd\tau\right)
 &=\Gamma(1-\alpha)s^\alpha\widehat f(s).
 \label{eq:Lap-fractional}
\end{align}
The second formula follows by Fubini and the change of variables $z=s(t-\tau)$.

The implication
\[
 f(t)\approx t^{\alpha-1}\quad(t\to\infty)
 \quad\Longrightarrow\quad
 \widehat f(s)\approx s^{-\alpha}\quad(s\downarrow0)
\]
is immediate. The reverse implication requires a Tauberian condition for pointwise asymptotics. Positivity alone still gives an integrated statement.

\begin{lemma}[Integrated decay from a Laplace bound]
\label{lem:integrated-tauberian}
Let $f(t)\ge0$ and suppose that, for $0<s\le1$,
\[
 c_-s^{-\alpha}\le\widehat f(s)\le c_+s^{-\alpha},
 \qquad 0\le\alpha\le1.
\]
Then, for every $\delta>0$, there exist positive constants $C_\pm$, independent of $\delta$ for $\delta$ in a fixed bounded interval, such that
\begin{equation}
 \frac{C_-}{\delta}
 \le\int_0^\infty f(t)t^{-\alpha-\delta}\,\dd t,
 \qquad
 \int_1^\infty f(t)t^{-\alpha-\delta}\,\dd t
 \le\frac{C_+}{\delta}.
 \label{eq:integrated-tauberian}
\end{equation}
\end{lemma}

\begin{proof}
Multiply the Laplace bounds by $s^{\alpha+\delta-1}$ and integrate over $s\in(0,1)$. The Fubini theorem gives
\[
 \int_0^1s^{\alpha+\delta-1}\widehat f(s)\,\dd s
 =\int_0^\infty f(t)t^{-\alpha-\delta}
 \left(\int_0^t z^{\alpha+\delta-1}e^{-z}\,\dd z\right)\dd t.
\]
The left-hand side is bounded above and below by constants times $1/\delta$. The incomplete gamma factor is bounded above uniformly and is bounded below by a positive constant for $t\ge1$, which gives \eqref{eq:integrated-tauberian}.
\end{proof}

\section*{Declarations}
\noindent\textbf{Conflict of interest.} The authors declare no conflict of interest.

\noindent\textbf{AI-assisted tools.} Used during manuscript preparation to improve the proof organizations and language of the manuscript.

\noindent\textbf{Data availability.} No data have been used or produced.

\end{document}